\documentclass{amsart}
\usepackage{amssymb,amscd}
\usepackage{graphicx}

\usepackage{latexsym}
\usepackage{multirow}
\usepackage{setspace}

\usepackage{amsthm}
\newtheorem{theorem}{Theorem}

\newtheorem{defn}{Definition}
\newtheorem{lemma}{Lemma}
\newtheorem{cor}{Corollary}

\newtheorem{thm}{Theorem}[section]

\newcommand{\bc}{\mathbb{C}}

\newcommand{\bp}{\mathbb{P}}
\newcommand{\bq}{\mathbb{Q}}
\newcommand{\br}{\mathbb{R}}
\newcommand{\bz}{\mathbb{Z}}
\newcommand{\modm}{\mathcal{M}}

\newcommand{\f}{\mathcal{F}}

\newcommand{\ie}{i.e.\ }
\newcommand{\fat}{\f\hspace{-.3mm}{\rm at}_{g,n}}
\newcommand{\fato}{\f\hspace{-.3mm}{\rm at}}

\begin{document}

\title{Counting lattice points in the moduli space of curves.}
\author{Paul Norbury}
\address{Department of Mathematics and Statistics\\
University of Melbourne\\Australia 3010\\
and Boston University\\
111 Cummington St\\
Boston MA 02215}
\email{pnorbury@ms.unimelb.edu.au}
\keywords{}
\subjclass{MSC (2000) 32G15; 11P21; 57R20}
\date{}

\begin{abstract}

\noindent We show how to define and count lattice points in the moduli space $\modm_{g,n}$ of genus $g$ curves with $n$ labeled points.  This produces a polynomial with coefficients that include the Euler characteristic of the moduli space, and tautological intersection numbers on the compactified moduli space.

\end{abstract}

\maketitle

\section{Introduction}

Let $\modm_{g,n}$ be the moduli space of genus $g$ curves with $n$ labeled points.  The {\em decorated} moduli space $\modm_{g,n}\times\br^n_+$ equips the labeled points with positive numbers $(b_1,...,b_n)$.   It has a cell decomposition due to Penner, Harer, Mumford and Thurston 
\begin{equation}  \label{eq:cell}
\modm_{g,n}\times\br^n_+\cong\bigcup_{\Gamma\in \fat}P_{\Gamma}
\end{equation}
where the indexing set $\fat$ is the space of labeled fatgraphs of genus $g$ and $n$ boundary components.  See Section~\ref{sec:fat} for definitions of a fatgraph $\Gamma$, its automorphism group $Aut \Gamma$ and the cell decomposition (\ref{eq:cell}) realised as the space of labeled fatgraphs with metrics.  Restricting this homeomorphism to a fixed $n$-tuple of positive numbers $(b_1,...,b_n)$ yields a space homeomorphic to $\modm_{g,n}$ decomposed into compact convex polytopes $P_{\Gamma}(b_1,...,b_n)$.  When the $b_i$ are positive integers the polytope $P_{\Gamma}(b_1,...,b_n)$ is an integral polytope and we define $N_{\Gamma}(b_1,...,b_n)$ to be its number of positive integer points.  The weighted sum of $N_{\Gamma}$ over all labeled fatgraphs of genus $g$ and $n$ boundary components is the lattice count polynomial:
\begin{defn}   \label{th:lcp}
$\quad \displaystyle N_{g,n}(b_1,...,b_n)=\sum_{\Gamma\in \fat}\frac{1}{|Aut \Gamma|}N_{\Gamma}(b_1,...,b_n)$
\end{defn}
Each integral point in the polytope $P_{\Gamma}(b_1,...,b_n)$ corresponds to a Dessin d'enfants defined by Grothendieck \cite{GroEsq} which represents a curve in $\modm_{g,n}$ defined over $\bar{\bq}$.  Thus the lattice count polynomial $N_{g,n}(b_1,...,b_n)$ counts curves defined over $\bar{\bq}$.  This is described in Section~\ref{sec:fat} where the integral points in $P_{\Gamma}(b_1,...,b_n)$ represent metrics on labeled fatgraphs with integer edge lengths, or equivalently curves equipped with a canonical meromorphic quadratic (Strebel) differential with integral residues.  

Quite generally the number of integer points in a convex polytope is a piecewise defined polynomial.  Nevertheless the following theorem shows that a weighted sum of the piecewise defined polynomials $N_{\Gamma}(b_1,...,b_n)$ is a polynomial.
\begin{theorem}   \label{th:poly}
The number of lattice points $N_{g,n}(b_1,...,b_n)$ is a degree $3g-3+n$ polynomial in the integers $(b_1^2,...,b_n^2)$ depending on the parity of the $b_i$.
\end{theorem}
The dependence on the parity means that $N_{g,n}(b_1,...,b_n)$ is represented by $2^n$ polynomials (by symmetry at most $[\frac{n}{2}]+2$ are different.)  The polynomials are symmetric under permutations of $b_i$ of the same parity.  If the number of odd $b_i$ is odd then $N_{g,n}(b_1,...,b_n)=0$.  Otherwise, the top degree homogeneous part of $N_{g,n}(b_1,...,b_n)$ is independent of the parity.  Table~\ref{tab:poly} shows the simplest polynomials.  The factorisations are expected from the vanishing result of Lemma~\ref{th:van} in Section~\ref{sec:van}.

\begin{table}[b]  \label{tab:poly}
\caption{Lattice count polynomials for even $b_i$}
\begin{spacing}{1.4}  
\begin{tabular}{||l|c|c||} 
\hline\hline

{\bf g} &{\bf n}&$N_{g,n}(b_1,...,b_n)$\\ \hline

0&3&1\\ \hline
1&1&$\frac{1}{48}\left(b_1^2-4\right)$\\ \hline
0&4&$\frac{1}{4}\left(b_1^2+b_2^2+b_3^2+b_4^2-4\right)$\\ \hline
1&2&$\frac{1}{384}\left(b_1^2+b_2^2-4\right)\left(b_1^2+b_2^2-8\right)$\\ \hline
2&1&$\frac{1}{2^{16}3^35}\left(b_1^2-4\right)\left(b_1^2-16\right)\left(b_1^2-36\right)\left(5b_1^2-32\right)$\\
\hline\hline
\end{tabular} 
\end{spacing}
\end{table}

Harer and Zagier \cite{HZaEul} calculated the orbifold Euler characteristic $\chi\left(\modm_{g,1}\right)$ and Penner \cite{PenPer} calculated $\chi\left(\modm_{g,n}\right)$ for general $n$.  This information is encoded in the lattice count polynomial for all even $b_i$.
\begin{theorem}  \label{th:euler}
$N_{g,n}(0,...,0)=\chi\left(\modm_{g,n}\right)$.
\end{theorem}

Kontsevich \cite{KonInt} defined the volume polynomial
\[ V_{g,n}(b_1,...,b_n)=\sum_{\Gamma\in \fat}\frac{1}{|Aut \Gamma|}Vol_{\Gamma}(b_1,...,b_n)\]
where $Vol_{\Gamma}(b_1,...,b_n)$ is the volume of the convex polytope $P_{\Gamma}(b_1,...,b_n)$.  (The Laplace transform of $V_{g,n}$ appears as $I_g$ in \cite{KonInt}.)  He showed that the coefficients give intersection numbers of Chern classes of the tautological line bundles $L_i$ over the compactified moduli space $\overline{\modm}_{g,n}$.  By considering finer and finer meshes it follows that the homogeneous top degree part of the lattice point count polynomial is the volume polynomial. 
\begin{theorem}
$N_{g,n}(b_1,...,b_n)=V_{g,n}(b_1,...,b_n)+$ lower order terms.
\end{theorem}
\begin{cor}  \label{th:int}
For $|{\bf d}|=\sum_id_i=3g-3+n$ and ${\bf d}!=\prod d_i!$ the coefficient $c_{\bf d}$ of $b^{2{\bf d}}=\prod b_i^{2d_i}$ in $N_{g,n}(b_1,...,b_n)$ is the intersection number
\[c_{\bf d}=\frac{1}{2^{6g-6+2n-g}{\bf d}!}\int_{\overline{\modm}_{g,n}}c_1(L_1)^{d_1}...c_1(L_n)^{d_n}.\]
\end{cor}
Kontsevich proved that these tautological intersection numbers satisfy a recursion relation conjectured by Witten \cite{WitTwo} that determine the intersection numbers.  The lattice count polynomials satisfy a recursion relation that uniquely determine the polynomials and when restricted to the top degree terms imply Witten's recursion.

\begin{theorem}   \label{th:recurs}
The lattice count polynomials satisfy the following recursion relation which determines the polynomials uniquely from $N_{0,3}$ and $N_{1,1}$.
{\setlength\arraycolsep{2pt} 
\begin{eqnarray}    \label{eq:rec}
\left(\sum_{i=1}^nb_i\right)N_{g,n}(b_1,...,b_n)&=&\sum_{i\neq j}\sum_{p+q=b_i+b_j}pqN_{g,n-1}(p,b_1,..,\hat{b}_i,..,\hat{b}_j,..,b_n)\nonumber\\
&+&\sum_i\sum_{p+q+r=b_i} pqr\biggl[N_{g-1,n+1}(p,q,b_1,..,\hat{b}_i,..,b_n)\\
&&\hspace{20mm}+\hspace{-7mm}\sum_{\begin{array}{c}_{g_1+g_2=g}\\_{I\sqcup J=\{1,..,\hat{i},..,n\}}\end{array}}\hspace{-8mm}N_{g_1,|I|}(p,b_I)N_{g_2,|J|}(q,b_J)\biggr]\nonumber
\end{eqnarray}}
\end{theorem}
The proof of Theorem~\ref{th:recurs} is elementary.  The recursion relation (\ref{eq:rec}) is used to prove Theorem~\ref{th:poly}.  It resembles Mirzakhani's recursion relation \cite{MirSim} between polynomials giving the Weil-Petersson volume of the moduli space.  In fact the top homogeneous degree part of $N_{g,n}(b_1,...,b_n)$ coincides with the top homogeneous degree part of Mirzakhani's Weil-Petersson volume polynomial (after multiplying by an appropriate power of $2$) since both of these coincide with Kontsevich's volume.  Mirzakhani \cite{MirWei} already showed the coefficients of the Weil-Petersson volume polynomial are the intersection numbers given in Corollary~\ref{th:int}.  Do and Safnuk \cite{DSaHyp} use fatgraphs to give a simpler proof of Mirzakhani's recursion relation restricted to the top homogeneous degree part and show that it is a rescaled version of Mirzakhani's proof.   

Although Table~\ref{tab:poly} shows only even $b_i$, the recursion relation needs the odd cases too.  We will fill in the cases of odd $b_i$ here.  When $\sum b_i$ is odd, $N_{g,n}(b_1,...,b_n)\equiv 0$.  The polynomial $N_{0,4}(b_1,...,b_4)$ is the same as in the table when $b_1,...,b_4$ are all odd, and when exactly two of the $b_i$ are odd $N_{0,4}(b_1,...,b_4)=\frac{1}{4}\left(b_1^2+b_2^2+b_3^2+b_4^2-2\right)$.  For genus 1 when $b_1$ and $b_2$ are odd $N_{1,2}(b_1,b_2)=\frac{1}{384}\left(b_1^2+b_2^2-2\right)\left(b_1^2+b_2^2-10\right)$.

Section~\ref{sec:fat} contains preliminaries on fatgraphs and lattice point counting.  Theorems~\ref{th:poly} and \ref{th:recurs} are proven in Section~\ref{sec:rec}.  Section~\ref{sec:van} contains a simple vanishing result for $N_{g,n}(b_1,...,b_2)$ which has powerful consequences.  In Section~\ref{sec:euler} we prove Theorem~\ref{th:euler} and treat the special case of $n=1$ labeled points.  

{\em Acknowledgements.} The author would like to thank Norman Do for many useful conversations.

\section{Fatgraphs}  \label{sec:fat}

A {\em fatgraph} is a graph $\Gamma$ with vertices of valency $>2$ equipped with a cyclic ordering of edges at each vertex.  In Figure~\ref{fig:fat} we use the projection to define the cyclic ordering to be anticlockwise at each vertex.  
\begin{figure}[ht]  
	\centerline{\includegraphics[height=2.5cm]{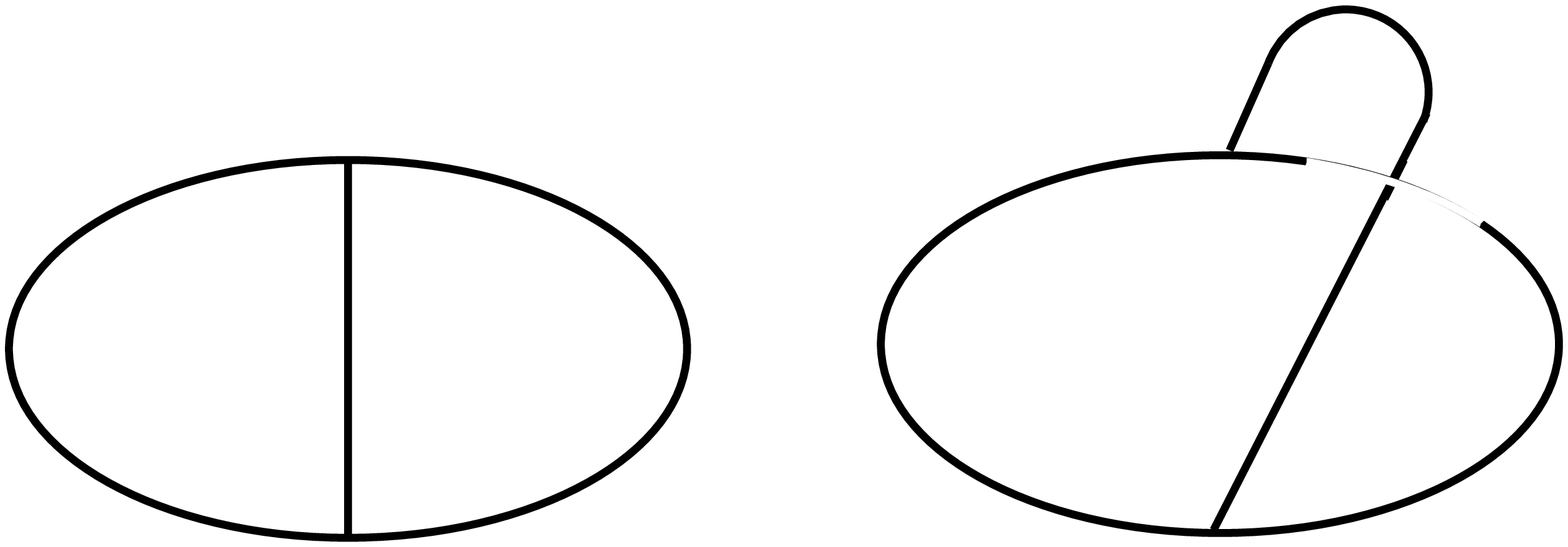}}
	\caption{Fatgraphs}
	\label{fig:fat}
\end{figure}
The two pictured fatgraphs are different, although the underlying graphs are the same.  A fatgraph structure on a graph is equivalent to an embedding of a graph into a surface $\Gamma\to\Sigma$ such that $\Sigma-\Gamma$ is a union of disks.  This gives a genus $g$ and number of boundary components $n$ to $\Gamma$.  The examples in Figure~\ref{fig:fat} have genus 0 and 1 shown in Figure~\ref{fig:fats}.  
\begin{figure}[ht]  
	\centerline{\includegraphics[height=2.5cm]{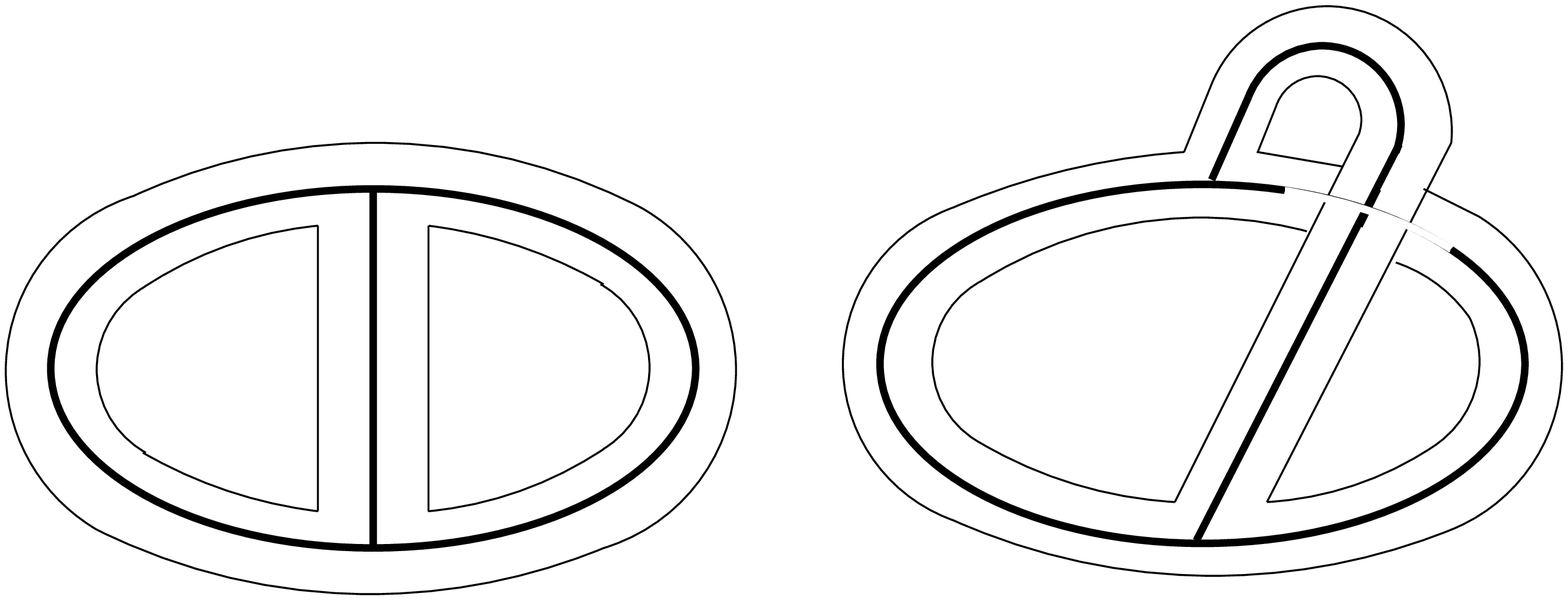}}
	\caption{Graphs embedded in genus 0 and 1 surfaces}
	\label{fig:fats}
\end{figure}

A {\em labeled} fatgraph is a fatgraph with boundary components labeled $1,...,n$.  The set of all labeled fatgraphs of genus $g$ and $n$ boundary components is notated by $\fat$.

It is useful to describe a fatgraph in the following equivalent way \cite{KonInt} which makes the automorphisms transparent.  Given a graph $\Gamma$ with vertices of valency $>2$, let $X$ be the set of oriented edges, so each edge of $\Gamma$ appears in $X$ twice.  Define the map $\tau_1:X\to X$ that flips the orientation of each edge.  A fatgraph, or ribbon, structure on $\Gamma$ is a map $\tau_0:X\to X$ that permutes cyclically the oriented edges with a common source vertex.  Let $X_0$, $X_1$ and $X_2$ be the vertices, edges and boundary components of the fatgraph $\Gamma$.  Then $X_0=X/\tau_0$, $X_1=X/\tau_1$ and $X_2=X/\tau_2$ for $\tau_2=\tau_0\tau_1$.  An {\em automorphism} of the labeled fatgraph $\Gamma$ is a permutation $\phi:X\to X$ that commutes with $\tau_0$ and $\tau_1$ and acts trivially on $X_2$.  The examples in Figure~\ref{fig:fat} given any labeling have automorphism groups $\{1\}$ and $\bz_6$.

A metric on a labeled fatgraph $\Gamma$ assigns positive numbers---lengths---to each edge of the fatgraph.  If $\Gamma\in \fat$ then the valency $>2$ conditions on the vertices ensures that the number of edges $e(\Gamma)$ of $\Gamma$ is bounded $e(\Gamma)\leq 6g-6+3n$.  Let $P_{\Gamma}$ be the $6g-6+3n$ cell consisting of all metrics on $\Gamma$.  Construct the cell-complex 
\[ \modm_{g,n}^{\rm combinatorial}=\bigcup_{\Gamma\in \fat}P_{\Gamma}\]
where we identify isometric metrics on fatgraphs, and when the length of an edge $l_E\to 0$ we identify this with the metric on the fatgraph with the edge $E$ contracted.  By the existence and uniqueness of meromorphic quadratic differentials with foliations having compact leaves, known as Strebel differentials, the cell complex is homeomorphic to the decorated moduli space $\modm_{g,n}^{\rm combinatorial}\cong\modm_{g,n}\times\br_+^n$ \cite{HarVir}.

Denote by $P_{\Gamma}(b_1,...,b_n)\subset P_{\Gamma}$ the metrics on $\Gamma$ with fixed boundary lengths ${\bf b}=(b_1,...,b_n)\in\br_+^n$ or equivalently with specified residues of the (square root of the) associated Strebel differential.  Then 
\begin{equation}  \label{eq:modcom}
\modm_{g,n}^{\rm combinatorial}(b_1,...,b_n)=\bigcup_{\Gamma\in \fat}P_{\Gamma}(b_1,...,b_n)\cong\modm_{g,n}.
\end{equation} 

\subsection{Counting lattice points in convex polytopes}

A convex polytope $P\subset\br^n$ can be defined as the convex hull of a finite set of vertices in $\br^n$.  We will consider {\em integral} polytopes $P$ where the vertices lie in $\bz^n$.  Define the number of integral points in $P$ by $N_P=\#\{P\cap\bz^n\}$ and $N_P(k)=\#\{kP\cap\bz^n\}$ where $kP$ rescales $\lambda_j\mapsto k\lambda_j$.  Also, define $N^0_P(k)$ to be the number of integral points in the {\em interior} of $kP$.
\begin{thm}[Ehrhart]  \label{th:ehrh}
If $P\subset\br^n$ is an $n$-dimensional convex polytope then
\[ N_P(k)={\rm Vol}(P)k^n+...\] 
is a degree $n$ polynomial in $k$ with top coefficient the volume of $P$.  Furthermore,
\[ N^0_P(k)=(-1)^nN_P(-k).\]
\end{thm}

We can define a convex polytope with positive codimension as follows.  Given a linear map $A:\br^N\to\br^n$ and ${\bf b}\in\br^n$ define 
\[ P_A({\bf b})=\{{\bf x}\in\br_+^N|A{\bf x}={\bf b}\}.\]  
If $A$ and $b$ have integer entries (with respect to the standard bases) then $P_A({\bf b})$ is integral and we define $N_{P_A}({\bf b})=\#\{P_A\cap\bz^N\}$.  In this case $N_{P_A}({\bf b})$ is a piecewise defined polynomial in $\bf b$ - for example, $N_{P_A}({\bf b})$ may be zero for some values of ${\bf b}$.

The set $P_{\Gamma}({\bf b})$ in (\ref{eq:modcom}) is a convex polytope defined by solutions ${\bf x}\in\br_+^{e(\Gamma)}$ of
\[ A_{\Gamma}{\bf x}={\bf b}\]
where $A_{\Gamma}$ is the incidence matrix that maps the vector space generated by edges of $\Gamma$ to the vector space generated by boundary components of $\Gamma$---an edge maps to the sum of its two incident boundary components.  In the examples in Figure~\ref{fig:fat} the incidence matrices are
\[ A_{\Gamma}=\left(\begin{array}{ccc}1&1&0\\1&0&1\\0&1&1\end{array}\right),\quad A_{\Gamma'}=\left(\begin{array}{ccc}2&2&2\end{array}\right).\]
We define 
\[ N_{\Gamma}({\bf b})=\#\{P_A\cap\bz_+^N\}.\]
It is natural to allow non-negative solutions although we allow only {\em positive} integer solutions.  This is justified by the fact that if some of the $x_i$ vanish then this will be counted using a fatgraph obtained by collapsing edges of $\Gamma$.  (If the collapsing of edges of $\Gamma$ does not yield a fatgraph, for example collapsing a loop, then we do not want to count such solutions.)

Since each edge is incident to exactly two (not necessarly distinct) boundary components the columns of $A_{\Gamma}$ add to 2, or equivalently $(1,1,...,1)\cdot A_{\Gamma}=(2,2,...,2)$.  Thus, 
\[\sum b_i=(1,1,...,1)\cdot{\bf b}=(1,1,...,1)\cdot A_{\Gamma}{\bf x}=(2,2,...,2)\cdot{\bf x}\in 2\bz\] 
so $N_{\Gamma}({\bf b})=0$ if $\sum b_i$ is odd.  Hence the lattice count polynomial $N_{g,n}(b_1,...,b_n)$ given in Definition~\ref{th:lcp}
also vanishes when $\sum b_i$ is odd.

If we relax the condition on fatgraphs that the valency of each vertex must be $>2$ then Grothendieck \cite{GroEsq} showed that fatgraphs with all edge lengths 1 possess branched covers of $\bp^1$ branched over 0, 1 and $\infty$.  By a theorem of Belyi these correspond to curves defined over $\bar{\bq}$.  When the length of each edge is a positive integer this is the same as a string of length 1 edges joined by valency 2 vertices.  Thus, $N_{g,n}(b_1,...,b_n)$ counts curves defined over $\bar{\bq}$ branched over of $0,1,\infty\in\bp^1$ with all points over 1 of ramification 2, and all points over 0 of ramification $>2$.

For a convex polytope $P\subset\br^N$ and a polynomial $\phi$ on $\br^N$ define the following generalisation of counting lattice points.
\[ N_P(\phi,k)=\sum_{{\bf x}\in kP\cap\bz^N}\phi({\bf x})\]
and $N^0_P(\phi,k)$ the sum over interior integer points of $kP$.
Later when applying the recursion relation we will need to calculate sums with a parity restriction as in Lemma~\ref{th:sandr} because the polynomials $N_{g,n}$ vanish if the sum of the arguments is odd.
\begin{lemma}  \label{th:sandr}
\begin{equation}  \label{eq:sandr} 
S_m(k)=\sum_{\begin{array}{c}\scriptstyle p+q=k\\\scriptstyle q{\rm\ even}\end{array}}\hspace{-3mm}p^{2m+1}q,\quad R_{m,m'}(k)=\sum_{\begin{array}{c}\scriptstyle p+q+r=k\\\scriptstyle r{\rm\ even}\end{array}}\hspace{-5mm}p^{2m+1}q^{2m'+1}r
\end{equation}
are {\em odd} polynomials in $k$ of degree $2m+3$, respectively $2m+2m'+5$, depending on the parity of $k$.
\end{lemma}
\begin{proof}
The dependence on the parity means that there are two polynomials $S^{\rm even}_m(k)$ and $S^{\rm odd}_m(k)$ depending on whether $k$ is even or odd.  The same is true for $R_{m.m'}(k)$.  Notice that
\[ S_m(k)=2N_{P}(\phi_1,k)\]
for $P=\{(x,y)\in\br_+^2|x+2y=1\}$ and $\phi_1=x^{2m+1}y$ (substitute $q=2Q$.)  Similarly, 
\[ R_{m,m'}(k)=2N_{P'}(\phi_2,k)\]
for $P'=\{(x,y,z)\in\br_+^3|x+y+2z=1\}$ and $\phi_2=x^{2m+1}y^{2m'+1}z$.  

The polytopes $P$ and $P'$ are {\em rational}, not integral.  They
can be expressed in terms of the integral convex polytopes of higher dimension
\[ P_1=\{x\geq 0, y\geq 0, x+2y\leq 2\},\quad P_2=\{x\geq 0, y\geq 0,z\geq 0,x+y+2z\leq 2\}.\]
For $k$ even
\[ S^{\rm even}_m(k)=N_{P_1}(\phi_1,\frac{k}{2})-N^0_{P_1}(\phi_1,\frac{k}{2}),\quad R^{\rm even}_{m,m'}(k)=N_{P_2}(\phi_2,\frac{k}{2})-N^0_{P_2}(\phi_2,\frac{k}{2}).\]
A generalisation \cite{BVeLat} of Ehrhart's theorem states that for a dimension $n$ integral convex polytope $P\subset\br^n$, $N_P(\phi,k)$ is a degree $\deg\phi+n$ polynomial in $k$ and  
\[ N^0_P(\phi,k)=(-1)^{\deg\phi+n}N_P(\phi,-k).\]
For the cases at hand, $\deg\phi+n$ is even so the right hand side is $N_P(\phi,-k)$ and $S^{\rm even}_m(k)$ and $R^{\rm even}_{m,m'}(k)$ are {\em odd} polynomials in $k$ of degree $2m+3$, respectively $2m+2m'+5$.
For $k$ odd, 
\[ S^{\rm odd}_m(k)=N^0_{P_1}(\phi_1,\frac{k+1}{2})-N_{P_1}(\phi_1,\frac{k-1}{2})\]
and $R^{\rm odd}_{m,m'}(k)$ is the same expression with $P_2$ in place of $P_1$.  Once again $S^{\rm odd}_m(k)$ and $R^{\rm odd}_{m,m'}(k)$ are {\em odd} polynomials in $k$ of degree $2m+3$, respectively $2m+2m'+5$.
\end{proof}

\subsection{Recursion}  \label{sec:rec}
\begin{proof}[Proof of Theorem~\ref{th:recurs}]
The lattice count polynomial $N_{g,n}(b_1,...,b_n)$ counts labeled fatgraphs with positive integer edge lengths which we call integer fatgraphs in $P_{\Gamma}(b_1,...,b_n)$.  We can produce an integer fatgraph in $P_{\Gamma}(b_1,...,b_n)$ from simpler integer fatgraphs in the three ways shown in Figures~\ref{fig:rec1}, \ref{fig:rec2} and \ref{fig:rec3}.  Choose a graph in $P_{\Gamma'}(p,b_3,...,b_n)$ and add an edge of length $q/2$ inside the boundary of length $p$ as in Figure~\ref{fig:rec1} so that $p+q=b_1+b_2$.    
\begin{figure}[ht]  
	\centerline{\includegraphics[height=4cm]{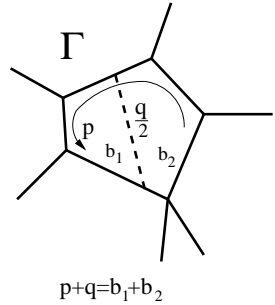}}
	\caption{$\Gamma$ is obtained from a simpler fatgraph by adding the broken line.}
	\label{fig:rec1}
\end{figure}
Similarly, attach an edge and a loop of total length $q/2$ inside the boundary of length $p$ as in Figure~\ref{fig:rec2} so that $p+q=b_1+b_2$.
\begin{figure}[ht]  
	\centerline{\includegraphics[height=4cm]{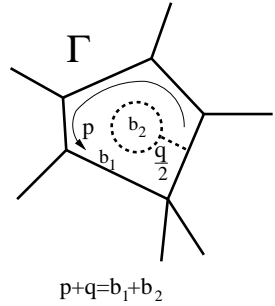}}
	\caption{$\Gamma$ is obtained by adding a line and loop of total length $q/2$.}
	\label{fig:rec2}
\end{figure}
In both cases for each $\Gamma'$ there are $p$ possible ways to attach the edge so this construction contributes $pN_{g,n-1}(p,b_3,...,b_n)$ to $N_{g,n}(b_1,...,b_n)$.  However we have overcounted, particularly when we repeat this construction for any pair $b_i$ and $b_j$, since each integer fatgraph in $P_{\Gamma}(b_1,...,b_n)$ can be produced in many ways like this.  To deal with this, we overcount even further by taking $pqN_{g,n-1}(p,b_3,...,b_n)$, \ie taking each constructed fatgraph $q$ times.  But now we see that for each edge that we attach of length $q/2$ we have overcounted $q$ times.  If we were to use all of the edges of $\Gamma$ in this way then we would have overcounted by
\[\sum_{E\in\Gamma}l(E)=\sum_{i=1}^nb_i.\]
Indeed all of the edges of $\Gamma$ are used, exactly once, when we include one further construction of the integer fatgraph $\Gamma$. 

Choose an integer fatgraph in $P_{\Gamma'}(p,q,b_2,...,b_n)$ for $\Gamma'\in\fato_{g-1,n+1}$ {\em or} choose two integer fatgraphs in $P_{\Gamma_1}(p,b_2,...,b_j)$ and $P_{\Gamma_2}(q,b_{j+1},...,b_n)$ for $\Gamma_1\in\fato_{g_1,j}$ and $\Gamma_2\in\fato_{g_2,n+1-j}$ where $g_1+g_2=g$ and attach an edge of length $r/2$ connecting these two boundary components as in Figure~\ref{fig:rec3} so that $p+q+r=b_1$.  
\begin{figure}[ht]  
	\centerline{\includegraphics[height=4cm]{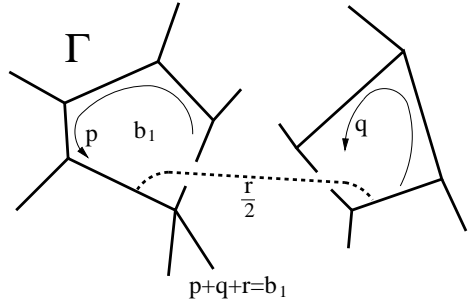}}
	\caption{$\Gamma$ is obtained from a single fatgraph or two disjoint fatgraphs by adding the broken line.}
	\label{fig:rec3}
\end{figure}

In the diagram, the two boundary components of lengths $p$ and $q$ are part of a fatgraph that may or may not be connected.
There are $pq$ possible ways to attach the edge so this construction contributes $pqN_{g-1,n+1}(p,q,b_2,...,b_n)$ and $pqN_{g_1,j}(p,b_2,...,b_j)N_{g_2,n+1-j}(q,b_{j+1},...,b_n)$ to $N_{g,n}(b_1,...,b_n)$ and again we have overcounted.  We overcount further by a factor of $r$ to get $pqrN_{g-1,n+1}(p,q,b_2,...,b_n)$ and $pqrN_{g_1,j}(p,b_2,...,b_j)N_{g_2,n+1-j}(q,b_{j+1},...,b_n)$.  We repeat this for each $g_1+g_2=g$ and $I\sqcup J=\{2,...n\}$ and then for each $b_j$ in place of $b_1$.

As previewed above, each edge of $\Gamma$ has been attached to construct $\Gamma$ and $N_{g,n}(b_1,...,b_n)$  has been overcounted $\sum_{i=1}^nb_i$ times yielding (\ref{eq:rec}).
\end{proof}
{\em Remark.}  The idea in the proof above to overcount by the length of each edge of the graph $\Gamma$ comes from the similar idea introduced by Mirzakhani \cite{MirSim} where she unfolds a function on Teichm\"uller space that sums to the analogue of $b_1$.

To apply the recursion we need to first calculate $N_{0,3}(b_1,b_2,b_3)$ and $N_{1,1}(b_1)$.  There are seven labeled fatgraphs in $\fato_{0,3}$ coming from three unlabeled fatgraphs.  It is easy to see that $N_{0,3}(b_1,b_2,b_3)=1$ if $b_1+b_2+b_3$ is even (and 0 otherwise.)  This is because for each $(b_1,b_2,b_3)$ there is exactly one of the seven labeled fatgraphs $\Gamma$ with a unique solution of $A_{\Gamma}{\bf x}={\bf b}$ while the other six labeled fatgraphs yield no solutions.  For example, if $b_1>b_2+b_3$ then only the fatgraph $\Gamma$ with $A_{\Gamma}=\left(\begin{array}{ccc}2&1&1\\0&1&0\\0&0&1\end{array}\right)$ has a solution and that solution is unique.

To calculate $N_{1,1}(b_1)$, note that $A_{\Gamma}=[2\quad 2\quad 2]$ or $[2\quad 2]$ for the 2-vertex and 1-vertex fatgraphs.  Hence
\[ N_{1,1}(b_1)=a_1\binom{\frac{b_1}{2}-1}{2}+a_2\binom{\frac{b_1}{2}-1}{1}\]
where $a_1$ is the number of trivalent fatgraphs (weighted by automorphisms) and $a_2$ is the number of 1-vertex fatgraphs.  The genus 1 graph $\Gamma$ from Figure~\ref{fig:fat} has $|Aut\Gamma|=6$ so $a_1=1/6$, and $a_2$ uses the genus 1 figure 8 fatgraph which has automorphism group $\bz_4$ hence $a_2=1/4$.  Thus
\[ N_{1,1}(b_1)=\frac{1}{6}\binom{\frac{b_1}{2}-1}{2}+\frac{1}{4}\binom{\frac{b_1}{2}-1}{1}=\frac{1}{48}\left(b_1^2-4\right).\]
We can also calculate $N_{1,1}(b_1)$ using a version of the recursion
\[ b_1N_{1,1}(b_1)=\frac{1}{2}\sum_{\begin{array}{c}p+q+p=b\\b{\rm\ even}\end{array}}pq.\]
We will calculate $N_{0,4}[b_1,b_2,b_3,b_4]$ to demonstrate the recursion relation and the parity issue.
\[\left(\sum_{i=1}^4b_i\right)N_{0,4}(b_1,b_2,b_3,b_4)=\sum_{i\neq j}\sum_{\begin{array}{c}\scriptstyle p+q=b_i+b_j\\\scriptstyle q{\rm\ even}\end{array}}pq.\]
If all $b_i$ are even, or all $b_i$ are odd, then $b_i+b_j$ is always even so the sum is over $p$ and $q$ even.  We have
\[ S^{\rm even}_0(k)=\sum_{i\neq j}\sum_{\begin{array}{c}\scriptstyle p+q=k\\\scriptstyle q{\rm\ even}\end{array}}pq=4\binom{\frac{k}{2}+1}{3}\]
so
\[\left(\sum_{i=1}^4b_i\right)N_{0,4}({\bf b})=\sum_{i\neq j}4\binom{\frac{b_i+b_j}{2}+1}{3}=\left(\sum_{i=1}^4b_i\right)\frac{1}{4}\left(b_1^2+b_2^2+b_3^2+b_4^2-4\right)\]
agreeing with Table~\ref{tab:poly}.  If $b_1$ and $b_2$ are odd and $b_3$ and $b_4$ are even then we need
\[ S^{\rm odd}_0(k)=\sum_{i\neq j}\sum_{\begin{array}{c}\scriptstyle p+q=k\\\scriptstyle q{\rm\ even}\end{array}}pq=\frac{1}{2}\binom{k+1}{3}\]
so
\begin{eqnarray*}
\left(\sum_{i=1}^4b_i\right)N_{0,4}({\bf b})&=&
\hspace{-8mm}\sum_{(i,j)=(1,2){\rm\ or\ }(3,4)}\hspace{-3mm}4\binom{\frac{b_i+b_j}{2}+1}{3}
+\hspace{-3mm}\sum_{(i,j)\neq(1,2){\rm\ or\ }(3,4)}\hspace{-1mm}\frac{1}{2}\binom{b_i+b_j+1}{3}\\
&=&\left(\sum_{i=1}^4b_i\right)\frac{1}{4}\left(b_1^2+b_2^2+b_3^2+b_4^2-2\right)
\end{eqnarray*}
so we see that the polynomial representatives of $N_{0,4}({\bf b})$ agree up to a constant term.

\begin{proof}[Proof of Theorem~\ref{th:poly}]
We can use the recursion (\ref{eq:rec}) to prove that $N_{g,n}(b_1,...,b_n)$ is a polynomial of the right degree but to prove that it is a polynomial in $b_i^2$ we need a different recursion formula (\ref{eq:rec1}).  For simplicity we use (\ref{eq:rec1}) to prove each part of Theorem~\ref{th:poly}.
{\setlength\arraycolsep{2pt} 
\begin{eqnarray}   \label{eq:rec1}
b_1N_{g,n}(b_1,...,b_n)&=&\sum_{j>1}\frac{1}{2}\left(\sum_{p+q=b_1+b_j}pqN_{g,n-1}(p,b_2,..,\hat{b}_j,..,b_n)\right.\\
&&\left.\quad\quad+\sum_{p+q=b_1-b_j}pqN_{g,n-1}(p,b_2,..,\hat{b}_j,..,b_n)\right)\nonumber\\
&+&\sum_{p+q+r=b_1} pqr\biggl[N_{g-1,n+1}(p,q,b_2,...,b_n)\nonumber\\
&&\hspace{20mm}+\hspace{-7mm}\sum_{\begin{array}{c}_{g_1+g_2=g}\\_{I\sqcup J=\{2,...,n\}}\end{array}}\hspace{-8mm}N_{g_1,|I|}(p,b_I)N_{g_2,|J|}(q,b_J)\biggr]\nonumber
\end{eqnarray}}
This differs from the recursion formula (\ref{eq:rec}) by breaking the symmetry around $b_1$.  The sum over the term $p+q=b_1-b_j$ needs to be interpreted as follows.  If $b_1-b_j>0$ it is read as written, whereas if $b_1-b_j<0$ then replace $b_1-b_j$ by $b_j-b_1$ and negate the sum.  (This is not the same as sending $(p,q)$ to $(-p,-q)$.)

We will prove the recursion (\ref{eq:rec1}) below.  Before that we will prove that given $N_{0,3}$ and $N_{1,1}$ then (\ref{eq:rec1}) determines polynomials $N'_{g,n}(b_1,...,b_n)$ of degree $3g-3+n$ in $b_i^2$.  By induction, the simpler polynomials are polynomials in $b_i^2$ so monomials on the right hand side of the recursion are of the form 
\[ S_m(k)=\sum_{\begin{array}{c}\scriptstyle p+q=k\\\scriptstyle q{\rm\ even}\end{array}}\hspace{-3mm}p^{2m+1}q,\quad R_{m,m'}(k)=\sum_{\begin{array}{c}\scriptstyle p+q+r=k\\\scriptstyle r{\rm\ even}\end{array}}\hspace{-5mm}p^{2m+1}q^{2m'+1}r\]
as in (\ref{eq:sandr}).  In Lemma~\ref{th:sandr} it is proven that $S_m(k)$ and $R_{m,m'}(k)$ are odd polynomials in $k$.  In particular, $S_m(b_1-b_j)=-S_m(b_j-b_1)$ explaining the interpretation of the sum over $b_1-b_j<0$

The sums over $p+q+r=b_1$ yield terms which are odd in $b_1$ from $R_{m,m'}(b_1)$ and even in $b_i$ for $i>1$ hence $1/b_1$ times these terms is even in all $b_i^2$.  The sums over $p+q=b_1+b_j$ and $p+q=b_1-b_j$ have the same summands so each monomial occurs with the same coefficient.  Hence the terms involving $b_1$ are of the form $S_m(b_1+b_j)+S_m(b_1-b_j)$ and since $S_m$ is odd, this sum is odd in $b_1$ and even in $b_j$, and even in the all other $b_i$.  Again $1/b_1$ times these terms is even in all $b_i^2$.  Thus by induction the polynomials generated by the recursion relation (\ref{eq:rec1}) from $N_{0,3}$ and $N_{1,1}$ are polynomials in $b_i^2$.

We will now calculate the degree in $b_i^2$.  By induction $\deg N_{g,n-1}=3g-3+n-1$ and by Lemma~\ref{th:van} $S_m(k)$ takes a term $p^{2m+1}q$ and produces a degree $2m+3$ polynomial, \ie it increases the degree by 1.  In this case $3g-3+n-1+1=3g-3+n$ as required.  Similarly, by induction $\deg N_{g-1,n+1}=3g-3+n-2$ and $\deg N_{g_1,|I|}N_{g_2,|J|}=3g-3+n-2$.  By Lemma~\ref{th:van} $R_{m,m'}(k)$ increases the degree of its summand by 2.  Since $3g-3+n-2+2=3g-3+n$ the result is proven by induction starting from the degrees of $N_{0,3}$ and $N_{1,1}$.\\

As above, write $N'_{g,n}$ for the polynomials produced from the recursion (\ref{eq:rec1}).  To prove the recursion (\ref{eq:rec1}) we use the fact that both (\ref{eq:rec}) and (\ref{eq:rec1}) uniquely determine $N_{g,n}$ and $N'_{g,n}$ respectively.  It remains to show that (\ref{eq:rec1}) $\Rightarrow$ (\ref{eq:rec}), hence $N_{g,n}$ and $N'_{g,n}$ necessarily coincide.

Apply (\ref{eq:rec1}) to each $b_i$ to calculate $b_iN'_{g,n}(b_1,...,b_n)$ and add.
{\setlength\arraycolsep{2pt} 
\begin{eqnarray*}
\left(\sum_{i=1}^nb_i\right)N'_{g,n}(b_1,...,b_n)&=&\sum_{i\neq j}\sum_{p+q=b_i+b_j}pqN'_{g,n-1}(p,b_1,..,\hat{b}_i,..,\hat{b}_j,..,b_n)\\
&+&\sum_i\sum_{p+q+r=b_i} pqr\biggl[N'_{g-1,n+1}(p,q,b_1,..,\hat{b}_i,..,b_n)\\
&&\hspace{20mm}+\hspace{-7mm}\sum_{\begin{array}{c}_{g_1+g_2=g}\\_{I\sqcup J=\{1,..,\hat{i},..,n\}}\end{array}}\hspace{-8mm}N'_{g_1,|I|}(p,b_I)N'_{g_2,|J|}(q,b_J)\biggr]\\
&+&\Delta
\end{eqnarray*}}
where
\[\Delta=\sum_{i\neq j}\frac{1}{2}\left(\sum_{p+q=b_i-b_j}+\sum_{p+q=b_j-b_i}\right)pqN'_{g,n-1}(p,b_1,..,\hat{b}_i,..,\hat{b}_j,..,b_n)=0\]
since the sums contain only canceling odd terms $S_m(b_i-b_j)+S_m(b_j-b_i)=0$.

Thus $N_{g,n}$ and $N'_{g,n}$ satisfy the recursion relation (\ref{eq:rec}) which uniquely determines them, hence 
\[ N_{g,n}=N'_{g,n}\]
so it follows that $N_{g,n}$ satisfies the recursion (\ref{eq:rec1}). 
\end{proof}
{\em Remark.} The top degree term of recursion (\ref{eq:rec1}) is a discrete version of the integration recursion for volume given by Do and Safnuk \cite{DSaHyp}.  They show their recursion is a rescaled version of Mirzakhani's recursion relation \cite{MirSim} which give the Virasoro relations among tautological classes \cite{MirWei}.

\subsection{Vanishing}   \label{sec:van}
\begin{lemma}   \label{th:van}
If $\displaystyle\sum_{i=1}^nb_i<4g+2n$ then $N_{g,n}(b_1,...,b_n)=0$.
\end{lemma}
\begin{proof}
A labeled fatgraph $\Gamma\in \fat$ has at least one vertex and hence at least $2g+n$ edges since $\chi(\Gamma)=1-2g-n$.  Since $N_{g,n}$ counts positive integers solutions of $A_{\Gamma}x=b$, each $x_i\geq 1$, thus $\sum x_i\geq 2g+n$.  Each edge contributes twice to the boundary of $\Gamma$ so 
\[ \sum_{i=1}^nb_i=2\sum_{i=1}^{e(\Gamma)}x_i\geq 4g+2n\]
and the lemma follows.
\end{proof}
Lemma~\ref{th:van} can be used to get strong information about the lattice count polynomial.  For example, $N_{1,1}(2)=0$ and since it is a polynomial in $b_1^2$ of degree 1 we get $N_{1,1}(b_1)=c(b_1^2-4)$.   The genus 2 case gives $N_{2,1}(2)=0=N_{2,1}(4)=N_{2,1}(6)$ hence
\[ N_{2,1}(b_1)=c_1(b_1^2-4)(b_1^2-16)(b_1^2-36)(b_1^2+c_2).\]

Although it is very difficult to calculate $N_{g,n}$ directly using fatgraphs, in the simplest cases it is calculable by extending the idea behind the vanishing Lemma~\ref{th:van}.  When $\displaystyle\sum_{i=1}^nb_i=4g+2n$ the argument in the proof of Lemma~\ref{th:van} shows that each $x_i=1$ so $N_{g,n}$ counts 1-vertex fatgraphs.

\section{Euler characteristic}  \label{sec:euler}
Using the cell decomposition (\ref{eq:cell}), the orbifold Euler characteristic of the moduli space can be calculated via a sum over labeled fatgraphs.  Expressing the sum as a Feynman expansion
Penner \cite{PenPer} calculated the following.
\[\chi(\modm_{g,n})=\sum_{\Gamma\in \fat}\frac{(-1)^{e(\Gamma)-n}}{|Aut \Gamma|}=\left\{\begin{array}{ll}(-1)^{n-1}(n-3)!&g=0\\(-1)^{n-1}\frac{(2g+n-3)!}{(2g-2)!}\zeta(1-2g)&g>0\end{array}\right.\]
The exponent is the dimension of the cell since $\dim P_{\Gamma}=e(\Gamma)-n$.

The lattice count polynomial gives another way to calculate the Euler characteristic via $N_{g,n}(0,...,0)=\chi\left(\modm_{g,n}\right)$.  We will prove this here.
\begin{proof}[Proof of Theorem~\ref{th:euler}.]
Define
\[ R_{g,n}(z)=\sum_{{\bf b}\in\bz_+^n}N_{g,n}(b_1,...,b_n)z^{b_1+...+b_n}.\]
It has the following properties:
\begin{enumerate}
\item $R_{g,n}(z)$ is a meromorphic function, holomorphic on $\bar{\bc}-\{\pm 1\}$.  \label{en:mer}
\item $R_{g,n}(0)=0$  \label{en:zer}
\item $R_{g,n}(\infty)=(-1)^nN_{g,n}(0,...,0)$.   \label{en:inf}
\end{enumerate}
Recall that $N_{g,n}({\bf b})$ is represented by a collection of polynomials depending on the parity of $b_i$.  By the symmetry of these polynomials we can set $R_{g,n}(z)=\sum_{k=0}^n\binom{n}{k}R^{(k)}_{g,n}(z)$ where $k=$ the number of odd $b_i$.  The basic idea behind property (\ref{en:mer}) is that if $p(n)=\sum_{j=0}^kp_jn^j$ is a polynomial then
\[\sum_{n>0}p(n)z^n=\sum_{j=0}^kp_j\sum_{n>0}n^jz^n=\sum_{j=0}^kp_j\left(z\frac{d}{dz}\right)^j\hspace{-2mm}\frac{z}{1-z}\]
which is a meromorphic function with pole at $z=1$ and known behaviour at $z=0$ and $z=\infty$.  If we restrict the parity of $n$ then
\begin{equation}   \label{eq:mero}
\sum_{\begin{array}{c}n>0\\n{\rm\ even}\end{array}}\hspace{-3mm}p(n)z^n=\sum_{j=0}^kp_j\left(z\frac{d}{dz}\right)^j\hspace{-2mm}\frac{z^2}{1-z^2},\sum_{\begin{array}{c}n>0\\n{\rm\ odd}\end{array}}\hspace{-3mm}p(n)z^n=\sum_{j=0}^kp_j\left(z\frac{d}{dz}\right)^j\hspace{-2mm}\frac{z}{1-z^2}
\end{equation}
which are both meromorphic functions with poles at $z=\pm 1$.  Furthermore, 
\[ \left(z\frac{d}{dz}\right)^j\hspace{-2mm}\left.\frac{z^2}{1-z^2}\right|_{z=\infty}=\left\{\begin{array}{cl}-1&j=0\\0&j>0\end{array}\right.,\quad
\left(z\frac{d}{dz}\right)^j\hspace{-2mm}\left.\frac{z}{1-z^2}\right|_{z=\infty}=0.\]
Each polynomial $N_{g,n}(b_1,...,b_n)$ is a sum of monomials of the form $\prod_1^nb_i^{2m_i}$ so $R^{(k)}_{g,n}(z)$ is a sum of finitely many series
\[R^{(k)}_{g,n}(z)=\sum_{\bf m}c_{\bf m}\hspace{-5mm}\sum_{\begin{array}{c}{\bf b}\in\bz_+^n\\b_i{\rm\ odd\ }i\leq k\end{array}}\hspace{-3mm}b_1^{2m_1}...b_n^{2m_n}z^{b_1+...+b_n}\]
which consists of terms of the form
\[\sigma^{(k)}_{\bf m}(z)=\hspace{-8mm}\sum_{\begin{array}{c}{\bf b}\in\bz_+^n\\b_i{\rm\ odd\ }i\leq k\end{array}}\hspace{-6mm}b_1^{2m_1}...b_n^{2m_n}z^{b_1+...+b_n}=\prod_{i=1}^k\hspace{-2mm}\sum_{\begin{array}{c}b_i>0\\b_i{\rm\ odd}\end{array}}\hspace{-3mm}b_i^{2m_i}z^{b_i}\cdot\prod_{i=k+1}^n\hspace{-2mm}\sum_{\begin{array}{c}b_i>0\\b_i{\rm\ even}\end{array}}\hspace{-3mm}b_i^{2m_i}z^{b_i}.\]
This is a finite product of meromorphic functions each with poles only at $z=\pm 1$ by (\ref{eq:mero}).  Furthermore, from the evaluation at $\infty$ of such functions, $\sigma^{(k)}_{\bf m}(\infty)=(-1)^n$ if ${\bf m}={\bf 0}$ and $k=0$ and it vanishes otherwise.  Thus, $R_{g,n}(\infty)$ contains only one non-vanishing term, $R_{g,n}(\infty)=R^{(0)}_{g,n}(\infty)=(-1)^nN_{g,n}(0,...,0)$ where we evaluate using the polynomial $N_{g,n}$ that takes in all even $b_i$.

We have proven (\ref{en:mer}) and (\ref{en:inf}).  Property (\ref{en:zer}) follows from the strict positivity of the $b_i$ and the convergence of the series which follows from the convergence of $1+z+z^2+...$ for $|z|<1$.

We can calculate $R_{g,n}(\infty)$ in another way. For a vector $v=(v_1,...,v_n)$ with $v_i\in\bz_+$ define (the semigroup homomorphism) $|v|=\sum_{i=1}^nv_i$.  Recall that the incident matrix $A_{\Gamma}=[\alpha_1,...,\alpha_{e(\Gamma)}]$ for $\alpha_i\in\br^n$ of a labeled fatgraph $\Gamma$ defines a convex polytope $A_{\Gamma}x={\bf b}$ and $N_{\Gamma}({\bf b})$ counts integral points $x\in\bz_+^{e(\Gamma)}$.  Thus
\begin{eqnarray*}
R_{\Gamma}(z)&=&\sum_{{\bf b}\in\bz_+^n}N_{\Gamma}(b_1,...,b_n)z^{b_1+...+b_n}=
\sum_{x\in\bz_+^{e(\Gamma)}}z^{|A_{\Gamma}x|}\\
&=&\sum_{x\in\bz_+^{e(\Gamma)}}z^{\sum_{i=1}^{e(\Gamma)}x_i|\alpha_i|}
=\prod_{i=1}^{e(\Gamma)}\sum_{x_i\in\bz_+}z^{x_i|\alpha_i|}\\
&=&\prod_{i=1}^{e(\Gamma)}\frac{z^{|\alpha_i|}}{1-z^{|\alpha_i|}}
\end{eqnarray*}
so $R_{\Gamma}(\infty)=(-1)^{e(\Gamma)}$ and
\[ R_{g,n}(\infty)=\sum_{\Gamma\in \fat}\frac{(-1)^{e(\Gamma)}}{|Aut\Gamma|}=(-1)^n\chi(\modm_{g,n}).\]
Combining this with property (\ref{en:inf}) yields the theorem
\[ N_{g,n}(0,...,0)=\chi(\modm_{g,n}).\]
\end{proof}

\subsection{Calculating $N_{g,1}$.} 
When $n=1$ there is a more direct proof of Theorem~\ref{th:euler}.  For any $\Gamma\in\fato_{g,1}$ the incidence matrix is $A_{\Gamma}=[2,2,...,2]$.  The equation $Ax=b$ has $\binom{\frac{b}{2}-1}{e(\Gamma)-1}$ positive integral solutions.  Hence
\[N_{g,1}(b)=c_{6g-3}^{(g)}\binom{\frac{b}{2}-1}{6g-4}+c_{6g-4}^{(g)}\binom{\frac{b}{2}-1}{6g-5}+..+c_k^{(g)}\binom{\frac{b}{2}-1}{k-1}+..+c_{2g}^{(g)}\binom{\frac{b}{2}-1}{2g-1}\]
where the coefficients are weighted counts of fatgraphs of genus $g$ with $n=1$ boundary component
\[ c_k^{(g)}=\sum_{\begin{array}{c}\Gamma\in\fato_{g,1}\\e(\Gamma)=k\end{array}}\frac{1}{|Aut\Gamma|}.\]
The polynomial $\binom{\frac{b}{2}-1}{k}$ evaluates at $b=0$ to $(-1)^k$ which gives a direct proof that the Euler characteristic is given by evaluation at 0.
\[N_{g,1}(0)=\sum_{\Gamma\in\fato_{g,1}}\frac{(-1)^{e(\Gamma)-1}}{|Aut\Gamma|}=\chi(\modm_{g,1}).\]

When $n=1$ the weighted number of trivalent fatgraphs and 1-vertex fatgraphs are known \cite{WLeCou}.
\[ c_{6g-3}^{(g)}=2\frac{1}{12^g}\frac{(6g-5)!}{g!(3g-3)!},\quad c_{2g}^{(g)}=\frac{(4g-1)!}{4^g(2g+1)!}\]
We can calculate $N_{2,1}(b)$ without using the recursion relation (except to deduce that $N_{2,1}(b)$ is a polynomial of degree $4$ in $b^2$) by applying Lemma~\ref{th:van} to get $N_{2,1}(b)=0$ for $b=2,4$ and 6.  This leaves two unknown coefficients which can be calculated from any two of the three pieces of known information $c_9^{(2)}$, $c_4^{(2)}$ and $N_{2,1}(0)$.
\begin{eqnarray*}
N_{2,1}(b)&=&\frac{1}{2^{16}3^35}\left(b^2-4\right)\left(b^2-16\right)\left(b^2-36\right)\left(5b^2-32\right)\\
&=&\textstyle\frac{35}{6}\binom{\frac{b}{2}-1}{8}+\frac{105}{4}\binom{\frac{b}{2}-1}{7}+\frac{93}{2}\binom{\frac{b}{2}-1}{6}+\frac{161}{4}\binom{\frac{b}{2}-1}{5}+\frac{84}{5}\binom{\frac{b}{2}-1}{4}+\frac{21}{8}\binom{\frac{b}{2}-1}{3}.
\end{eqnarray*}
The polynomial $N_{2,1}(b)$ enables us to calculate the weighted counts of fatgraphs $c_k^{(2)}$.  We can similarly calculate $N_{3,1}$ and hence deduce the weighted counts of fatgraphs.

\begin{eqnarray*}
N_{3,1}(b)&=&\textstyle\frac{1}{2^{25}3^65^27}(b^2\hspace{-1mm}-4)(b^2\hspace{-1mm}-16)(b^2\hspace{-1mm}-36)(b^2\hspace{-1mm}-64)(b^2\hspace{-1mm}-100)(5b^4\hspace{-1mm}-188b^2\hspace{-1mm}+1152)\\
&=&\textstyle\frac{5005}{3}\binom{\frac{b}{2}-1}{14}+\frac{25025}{2}\binom{\frac{b}{2}-1}{13}+41118\binom{\frac{b}{2}-1}{12}+\frac{929929}{12}\binom{\frac{b}{2}-1}{11}+\frac{183955}{2}\binom{\frac{b}{2}-1}{10}\\&&\textstyle+\frac{283767}{4}\binom{\frac{b}{2}-1}{9}+\frac{317735}{9}\binom{\frac{b}{2}-1}{8}+10813\binom{\frac{b}{2}-1}{7}+\frac{25443}{14}\binom{\frac{b}{2}-1}{6}+\frac{495}{4}\binom{\frac{b}{2}-1}{5}.
\end{eqnarray*}

\end{document}